%% file: symp.tex
\documentclass[11pt]{article}

\usepackage[numeric]{amsrefs}
\usepackage{amsfonts, amsmath, amssymb, amsxtra, stmaryrd}
\usepackage{amsthm}

\usepackage[english]{babel}

\usepackage{enumerate}
\usepackage{comment}
\usepackage{mathrsfs}
\usepackage{graphicx}
\usepackage{caption,subcaption, float}
\usepackage{mathtools}
\usepackage{lscape}
\usepackage{rotating}
\usepackage{array}
\usepackage{tikz}

\usetikzlibrary{decorations.markings}

\usepackage{hyperref}

\numberwithin{subsection}{section}
\numberwithin{equation}{section}
\theoremstyle{plain}
\newtheorem{satz}{Theorem}[section]
\newtheorem*{theorem*}{Theorem}
\newtheorem{lem}[satz]{Lemma}
\newtheorem{prop}[satz]{Proposition}

\newtheorem{cor}[satz]{Corollary}

\theoremstyle{definition}
\newtheorem{defi}[satz]{Definition}

\newtheorem{exam}[satz]{Example}

\newtheorem{remark}[satz]{Remark}

\allowdisplaybreaks

\newcommand{\MG}{\mathfrak{g}}
\newcommand{\g}{\mathfrak{g}}

\newcommand{\MF}{\mathbb{F}}
\newcommand{\MP}{\mathbb{P}}
\newcommand{\MSL}{\mathfrak{sl}_2}
\newcommand{\ME}{\mathcal{E}}
\newcommand{\E}{\mathcal{E}}

\newcommand{\MN}{\mathbb{N}}

\newcommand{\cV}{\mathcal{V}}
\newcommand{\chara}{\mathop{\mathrm{char}}}

\newcommand{\rad}{\mathop{\mathrm{rad}}}

\newcommand{\codim}{\mathop{\mathrm{codim}}}

\begin{document}

\include{yaelsymp}

\end{document}

%% file: yaelsymp.tex
\title{A geometric characterization of the symplectic Lie algebra }
\author{Hans Cuypers\footnote{corresponding author}, Yael Fleischmann \footnote{present address: 
Institut f\"ur Mathematik,
Warburger Str. 100,
33098 Paderborn,  yael.fleischmann@math.uni-paderborn.de
}\ \footnote{This work is part of the research programme ”Special elements in Lie Algebras”
(613.000.905), which is (partly) financed by the Netherlands Organisation for
Scientific Research (NWO).
}\\
Department of Mathematics and Computer Science\\
Eindhoven University of Technology\\
P.O. Box 513, 5600 MB Eindhoven\\
The Netherlands\\
email: f.g.m.t.cuypers@tue.nl}

\maketitle

\section{Introduction}
An element $x\neq 0$ in a Lie algebra $\mathfrak{g}$ over a field $\mathbb{F}$ of characteristic $\neq 2$ with Lie product $[\cdot,\cdot]$ is called {\em extremal} if for all $y\in \mathfrak{g}$ we have $$[x,[x,y]]\in\mathbb{F}x.$$

Arjeh Cohen, \emph{et al.} \cite{CSUW01} started the investigation of
Lie algebras generated by extremal elements in order  
to provide a geometric characterization of the classical Lie algebras. 
In \cite{CI07}, he and Gabor Ivanyos construct a point-line space 
on the set of nonzero extremal elements of $\mathfrak{g}$ 
called the \emph{extremal geometry} and denoted by $\E(\g)$.
The points of this geometry are the $1$-spaces generated by  extremal elements, called {\em extremal points}. A line is then the set of extremal points
inside a $2$-dimensional subspace of $\g$, all whose nonzero vectors are extremal elements.
    
Combining the results in \cite{CI06}, \cite{CI07} and \cite{KS02}, 
one can conclude that if $\mathfrak{g}$ 
is of finite dimension, simple and generated by  extremal elements, then $\E(\mathfrak{g})$ either contains no lines or is the root shadow space of a spherical building (i.e., a point-line space related to a spherical building). 
If this building is of finite rank at least $3$, 
then, as is shown in \cite{CRS14, Fl2015,CuyFle17}, the Lie algebra 
$\mathfrak{g}$ is indeed a classical Lie algebra and its extremal elements are 
the long root elements.

In this paper we are concerned with the case that there are no lines 
in the extremal geometry.

An example of such a Lie algebra $\mathfrak{g}$ is the finitary 
symplectic Lie algebra $\mathfrak{fsp}(V,f)$, where  $(V,f)$ is a 
nondegenerate symplectic space over the field $\mathbb{F}$.
Indeed, this Lie algebra is simple, provided the characteristic of $\mathbb{F}$ is not $2$,  
and generated by its rank $1$ elements, which are extremal. 
(Here the rank refers to the dimension of the image of the element in its action on the  module $V$.)
As these rank $1$ elements are the only  extremal elements in $\mathfrak{fsp}(V,f)$  and the sum of any two independent rank $1$ elements is of rank $2$, we do not find any lines in $\E(\g)$. 
See Section \ref{sec:sympl}.

We provide the following characterization of these symplectic Lie algebras.

\begin{satz}\label{MainThm}
Let $\MG$ be a simple Lie algebra over the field $\MF$ with $\chara \MF\neq 2$ and generated by its set of  extremal elements $E$.
Assume the following:

\begin{enumerate}[\rm (a)]
\item  any two  extremal elements 
$x$ and $y$ in $\g$ either commute
or generate an $\mathfrak{sl}_2(\mathbb{F})$;
\item  for any three  extremal elements $x,y,z$ in $\g$ with $[x,y]\neq 0$,
there is an extremal $u$ in the subalgebra $\langle x,y\rangle$ commuting with $z$.
\end{enumerate}
 Then  $\MG$ is isomorphic to $\mathfrak{fsp}(V,f)$ 
for some nondegenerate symplectic space $(V,f)$.

Moreover, under this isomorphism the  extremal elements in $\g$ are mapped to rank $1$ elements in $\mathfrak{fsp}(V,f)$. 
\end{satz}

In Section \ref{sec:extremal} we will show that
the first condition is equivalent with   
the extremal geometry $\E(\g)$ not containing lines.
The second condition guarantees that even after a field extension 
the extremal geometry does  not contain lines.
Both conditions can easily be checked to hold  true in $\mathfrak{fsp}(V,f)$.

The paper is organized as follows.
In Section \ref{sec:extremal} we will provide some basic facts on extremal elements.
Section \ref{sec:sympl}  discusses the finitary symplectic Lie algebras and  their extremal elements.
In  Section \ref{sec:geom} we start with the proof of 
Theorem \ref{MainThm}, which is spread out over the sections \ref{sec:geom}
up to \ref{sec:final}. 

The proof is divided into three major steps.

First we construct a geometry on the set of extremal points in  
a Lie algebra $\g$ satisfying the conditions of Theorem \ref{MainThm}, called the $\MSL$-geometry, which turns out to be isomorphic to the geometry $HSp(V,f)$ of points and hyperbolic lines in a symplectic space $(V,f)$. This is done in Section \ref{sec:geom}.

In the second step, provided in Section \ref{sec:unique}, we show that 
the  Lie product of $\g$ is actually, up to multiplication with a nonzero scalar,
the unique Lie product on  the vector space of $\g$ 
inducing the  $\MSL$-geometry. 

As the $\MSL$-geometry of the symplectic Lie algebra $\mathfrak{fsp}(V,f)$ is
also isomorphic to the $HSp(V,f)$, this implies
that, in order  to prove Theorem \ref{MainThm}, it suffices to establish that
there is a semi-linear invertible map between the vector spaces of $\g$ and $\mathfrak{fsp}(V,f)$ inducing an isomorphism between the corresponding $\MSL$-geometries.

The existence of this semi-linear invertible map  is the topic of the sections \ref{sec:embedding} and  \ref{sec:final}. In Section \ref{sec:embedding} we treat the case where $\g$ is finite dimensional, while
 Section \ref{sec:final} is devoted to  the infinite dimensional case.

\bigskip

\noindent
{\bf Acknowledgment.} Parts of this paper can be found in the second author's PhD-thesis \cite{Fl2015}, which was written under supervision of Arjeh Cohen and the first author. We thank Arjeh Cohen for many inspiring discussions on the topic.

\section{Extremal elements}
\label{sec:extremal}

In this section we provide some 
results on extremal elements from \cite{CSUW01}, that 
will be very useful in the proof of our main results.

Let $\mathfrak{g}$
be a Lie algebra over the field $\mathbb{F}$ of characteristic $\neq 2$ and with
Lie bracket $[\cdot,\cdot]$. 
Then an {\em extremal} element of $\mathfrak{g}$ 
is a nonzero element $x\in \mathfrak{g}$ such that for all $y\in\mathfrak{g}$ we have
$$[x,[x,y]]\in \mathbb{F}x.$$

An  extremal element is called a {\em sandwich} if $[x,[x,y]]=0$ for all $y\in \mathfrak{g}$, and it is called {\em pure} if there is at least one $y\in \mathfrak{g}$ with $[x,[x,y]]\neq 0$.

By $E$ we denote the set  of all extremal elements in $\mathfrak{g}$.
An {\em extremal point} is a $1$-dimensional subspace of $\g$ spanned by an extremal element. The set of all extremal points in $\g$ is denoted by $\ME$.
We assume that $\mathfrak{g}$ is  generated by its set of extremal elements $E$.

\begin{prop}\cite[2.1]{CSUW01}\label{2generators}
For $x,y\in E$ we have one of the following:

\begin{enumerate}[\rm (a)]
\item $\mathbb{F}x=\mathbb{F}y$;
\item $[x,y]=0$ and $\lambda x+\mu y\in E\cup \{0\}$ for all $\lambda,\mu\in \mathbb{F}$;
\item $[x,y]=0$ and $\lambda x+\mu y\in E$ only if $\lambda=0$ or $\mu=0$;
\item $z:=[x,y]\in E$, and $x,z$ and $y,z$ are as in case {\rm (b)};
\item $\langle x,y\rangle $ is isomorphic to $\mathfrak{sl}_2(\mathbb{F})$.
\end{enumerate}

\end{prop}

\begin{prop}\cite[2.6]{CSUW01}\label{extremalform}
There is an associative symmetric bilinear form $g:\mathfrak{g}\times \mathfrak{g}\rightarrow\mathbb{F}$,
such that for all $x\in E$ and $y\in \mathfrak{g}$ we have
$$[x,[x,y]]=2g(x,y)x.$$
\end{prop}

The form $g$ {\em can and will} be chosen in such a way that for all sandwiches
$x\in E$ we have $g(x,y)=0$ for all $y$.
In particular, sandwiches are in the radical of $g$.

The form $g$ is called the {\em extremal form} on $\mathfrak{g}$.
As the form $g$ is associative, its radical $\rad(g)$ is an ideal in $\mathfrak{g}$. 

Notice that the extremal form $f$ from \cite{CSUW01} satisfies $f=2g$.

\begin{prop}\cite{CSUW01}\label{auto}
Let $x\in E$ be pure. Then for each $\lambda\in \MF$ the map
$$\mathrm{exp}(x,\lambda):\g\rightarrow \g,$$
defined by
$$\mathrm{exp}(x,\lambda)y=y+\lambda[x,y]+\lambda^2g(x,y)x$$
for all $y\in \g$, is an automorphism of $\g$.
\end{prop}

If $x,y\in E$ satisfy $g(x,y)\neq 0$, then the above proposition implies that
all elements of the form  $y+\lambda[x,y]+\lambda^2g(x,y)x$ are extremal.
The following result states that these are, up to scalar multiplication, the only extremal elements
in the subalgebra $\langle x,y\rangle$.
 
\begin{prop}\cite[3.1]{CSUW01}\label{elementsinsl2}\label{sl2lines}
Let $x,y\in E$ with $g(x,y)\neq 0$. Then $\langle x,y\rangle \simeq \mathfrak{sl}_2(\mathbb{F})$. 
Moreover, the extremal elements of $\mathfrak{g}$ in  $\langle x,y\rangle $ are the  elements

$$\alpha g(x,y)x+\beta y+\gamma [x,y],$$ where $\alpha\beta=\gamma^2$.

The subalgebra $\langle x,y\rangle$ is generated by any two  noncommuting extremal elements contained in it.
\end{prop}

\begin{remark}\label{oval}
Let $x,y\in E$ with $g(x,y)\neq 0$.
Then the above proposition implies that the 
extremal points of $\g$ in $\langle x,y\rangle$ are the points of a 
{\em conic} in the projective plane spanned by $E\cap\langle x,y\rangle$.
In particular, they form a {\em oval} in this plane, i.e., a set of points
such that any line of the plane meets it in at most two points, and any point of the conic is on exactly one line meeting the conic in just one point.
This line is called  the {\em tangent line}.
The tangent line through $\langle x\rangle$ is the projective line spanned by  $x$ and $[x,y]$.
\end{remark}

An {\em extremal line} in $\mathfrak{g}$ is a 2-dimensional subspace of $\mathfrak{g}$
such that all its elements are extremal and pairwise commuting.
Two linearly independent elements on an extremal line are as in case (b) of the above Proposition \ref{2generators}.

So, we obtain the following.

\begin{lem}
If the Lie algebra  $\mathfrak{g}$ does not contain 
extremal lines, then any two of its extremal elements either
commute or generate an $\mathfrak{sl}_2(\mathbb{F})$.
\end{lem}

\begin{lem}\label{quadratic}
Let $x,y$ and $z\in E$ be linearly independent and such that  $g(x,y)$, $g(x,z)$ and $g(y,z)$ are all nonzero.
If there is no extremal element $u\in\langle x,y\rangle$ commuting with $z$,
then there exists a quadratic extension $\hat{\mathbb{F}}$ of $\mathbb{F}$
such that  $\mathfrak{g}\otimes_{\mathbb{F}} \hat{\mathbb{F}}$
contains extremal lines.
\end{lem}

\begin{proof}
Suppose $x,y$ and $z\in E$ such that  $g(x,y)$, $g(x,z)$ and $g(y,z)$ are all nonzero.
Moreover, assume that there is no extremal element $u\in\langle x,y\rangle$ commuting with $z$.
By \ref{elementsinsl2}, we find the elements $u_\lambda=g(x,y)x+\lambda^2 y+\lambda [x,y]$, where $\lambda\in \mathbb{F}$, to be extremal.
Now $$g(z,u_\lambda)=g(x,y)g(z,x)+\lambda^2 g(z,y)+\lambda g(z,[x,y])$$
either takes the value $0$ for some value in $\lambda=\lambda_1\in \mathbb{F}$ and we find that $u_{\lambda_1}$
does not commute with $z$ but $g(z,u_{\lambda_1})=0$ (case (d) of Proposition \ref{2generators}), which implies that there are extremal lines in $\mathfrak{g}$,
or we find two distinct elements $\lambda_1$ and $\lambda_2$ in a quadratic extension  $\hat{\mathbb{F}}$ of $\mathbb{F}$ with
$g(z,u_{\lambda_1})=g(z,u_{\lambda_2})=0$. 

Suppose we are in the latter case. Then inside $\mathfrak{g}\otimes_{\mathbb{F}} \hat{\mathbb{F}}$ we find the following.
As $\langle u_{\lambda_1},u_{\lambda_2}\rangle$ contains $\langle x,y\rangle$, the element $z$ cannot commute with both
$ u_{\lambda_1}$ and $u_{\lambda_2}$, which then implies that $z$ and at least one of $u_{\lambda_1}$ and $u_{\lambda_2}$ are in relation (d) of \ref{2generators}.
But then $\mathfrak{g}\otimes_{\mathbb{F}} \hat{\mathbb{F}}$ contains extremal lines.
\end{proof}

As Cohen {\em et al.} have studied Lie algebras generated by extremal elements containing  extremal lines, the 
two lemmas justify that, in order to study those Lie algebra generated by extremal elements that do not have extremal lines,
we can restrict ourselves to the Lie algebras satisfying conditions (a) and (b) of Theorem \ref{MainThm}.

\section{The symplectic Lie algebra}
\label{sec:sympl}

We begin this section with a description of the (finitary) symplectic Lie algebra in terms of  tensors. Using this, we provide a  description of the extremal elements of this (finitary) symplectic Lie algebra.

Let $V$ be a vector space over a field $\mathbb{F}$ with dual space $V^*$.
Then $V\otimes V^*$ is isomorphic to the linear space of {\em finitary} linear maps from $V$ to itself. (A linear map is finitary if its kernel has finite codimension.)
Indeed, a pure tensor $v\otimes \phi$ acts linearly on $V$ by $$w\in V\mapsto v\phi(w)$$ and its kernel has codimension one.
As any nonzero element $x\in V\otimes V^*$ can be written as
a finite sum $x=v_1\otimes \phi_1+\cdots +v_k\otimes \phi_k$ with $v_1,\dots, v_k$ and $\phi_1,\dots,\phi_k$ linearly independent,
the action of a nonzero $x$ on $V$ is nontrivial with a kernel of codimension  $k$.   
Moreover, any finitary linear map can be realized by a finite sum of pure tensors.

On $V\otimes V^*$ we can define a Lie bracket by linear extension of  the following product for pure tensors $v\otimes \phi$ and $w\otimes \psi$:
\begin{align*}
[v\otimes\phi, w\otimes \psi] =& (v\otimes \psi)\phi(w)-(w\otimes \phi)\psi(v).
\end{align*}

The corresponding Lie algebra we denote by $\g(V\otimes V^*)$. From the above we easily deduce that $\g(V\otimes V^*)$ is isomorphic with $\mathfrak{fgl}(V)$, the finitary general linear Lie algebra.

For $x\in \mathfrak{g}(V\otimes V^*)$ and $v\in V$  we will write $x(v)$ for the image of $v$ under the natural action of $x$ on $V$.

So, for all $x,y\in \mathfrak{g}(V\otimes V^*)$ and $v\in V$
we do have $$[x,y](v)=x(y(v))-y(x(v)).$$

Suppose $f$ is a bilinear form on $V$. Then for each vector $v\in V$ the map
$f_v:V\rightarrow \mathbb{F}$, with $f_v(w)=f(v,w)$ for all $w\in V$
is an element of $V^*$.
By $S_f(V\otimes V^*)$ or just $S_f$ we denote the subspace of $V\otimes V^*$
spanned by the pure tensors $v\otimes f_v$, with $v\in V$.
(It is spanned by the $f$-symmetric elements.)

We are interested in the case that the bilinear form $f$ is alternating (or symplectic).
So, let $(V,f)$ be a  symplectic space, that is the vector space $V$ together with a nontrivial alternating bilinear form $f:V\times V\rightarrow \mathbb{F}$. 
Moreover, assume that the characteristic of the field $\mathbb{F}$ is different from $2$.

The space ${S}_f(V\otimes V^*)$ is closed under the Lie bracket:
\begin{align*}
[v\otimes f_v, w\otimes f_w] =& (v\otimes f_w)f(v,w)-(w\otimes f_v)f(w,v)\\
=& f(v,w)(v\otimes f_w+w\otimes f_v)\\
=& f(v,w) [(v+w) \otimes f_{v+w} -v\otimes f_v -w\otimes f_w].\\
\end{align*}

The corresponding  Lie subalgebra of $\mathfrak{g}(V\otimes V^*)$ will be denoted by  $\mathfrak{s}_f(V\otimes V^*)$ or, for short, $\mathfrak{s}_f$.

If $f$ is nondegenerate, then  $\mathfrak{s}_f$  can be identified with  the finitary Lie algebra $\mathfrak{fsp}(V,f)$, i.e., the Lie subalgebra of $\mathfrak{fgl}(V)$
of finitary linear transformations $t:V\rightarrow V$ satisfying
$f(t(v),w)=-f(v,t(w))$ for all $v,w\in V$.

\begin{satz}
Suppose $(V,f)$ is a nondegenerate symplectic space over the field $\mathbb{F}$
of characteristic $\neq 2$.
Then the  Lie algebra $\mathfrak{s}_f$ is isomorphic to $\mathfrak{fsp}(V,f)$.
\end{satz}

\begin{proof}
For all $v,u,w\in V$ we have
$$f((v\otimes f_v)(u),w)=f(f(v,u)v,w)=f(v,u)f(v,w)$$
and 
$$f(u,(v\otimes f_v)(w)=f(u,f(v,w)v)=f(v,w)f(u,v).$$
So, $f((v\otimes f_v)(u),w)=-f((v,(v\otimes f_v)w)$.
This shows that the elements from $\mathfrak{s}_f$ induce symplectic elements on
$V$.
As the action of $V\otimes V^*$ on $V$ is faithful, we find that $\mathfrak{s}_f$ is isomorphic to a  subalgebra of $\mathfrak{fsp}(V,f)$.
We will show that this subalgebra equals  $\mathfrak{fsp}(V,f)$.
Hereby we identify elements of $\mathfrak{s}$ with their images in $\mathfrak{fsp}(V,f)$.

Suppose $s\in \mathfrak{fsp}(V,f)$.
We prove by induction on the codimension of $\ker(s)$, that $s\in \mathfrak{s}_f$.

If $\codim (\ker(s))=0$, then there is nothing to prove.

Now suppose $\codim (\ker(s))\geq 1$.

First we prove that there is an element $v\in V$ with $f(v,s(v))\neq 0$.
Let $v_1\in V$ with $s(v_1)\neq 0$. We can assume that $f(v_1,s(v_1))=0$.
Now take $v_2\in V$ with $f(s(v_1),v_2)=1$.
We can assume that also $f(v_2,s(v_2))=0$.
But then 
$$\begin{array}{rl}
f(v_1+v_2,s(v_1+v_2))
=& f(v_1,s(v_1))+f(v_1,s(v_2))+\\
&+f(v_2,s(v_1))+f(v_2,s(v_2))\\
=& 0-f(s(v_1),v_2)-f(s(v_1),v_2)+0\\
=& -2f(s(v_1),v_2)=-2\neq 0.\\
\end{array}$$
So, indeed, there is an element $v\in V$ with $f(v,s(v))\neq 0$.
Fix such element $v$ and let $w=s(v)\neq 0$. Then
$$(s+f(v,w)^{-1} w\otimes f_w)(v)=w+f(v,w)^{-1} f(w,v)w=0.$$
Moreover, if $u\in \ker(s)$, then $f(u,w)=f(u,s(v))=-f(s(u),v)=0$
and hence $(s+f(v,w)^{-1} w\otimes f_w)(u)=0$.
So, $\ker (s+f(v,w)^{-1} w\otimes f_w)$ contains $\ker(s)$ as a proper subspace.
This implies  by
induction that $s+f(v,w)^{-1} w\otimes f_w\in \mathfrak{s}_f$. But then,
as $f(v,w)^{-1} w\otimes f_w\in \mathfrak{s}_f$, also
$s\in \mathfrak{s}_f$. 
\end{proof}

\begin{lem}\label{rank1extremal}
\begin{enumerate}[$(a)$]
\item The pure tensors $v\otimes f_v$, where $0\neq v\in V$, are extremal in $\mathfrak{s}_f$.

\item The  extremal element $v\otimes f_v$ is  pure  if and only if $v\not\in\rad(f)$.
\item The extremal form $g$ satisfies $g(v\otimes f_v,w\otimes f_w)=f(v,w)^2$.
\end{enumerate}
\end{lem}

\begin{proof}
For $v,w\in V$ different from $0$ we have
$$\begin{array}{l}
[v\otimes f_v,[v\otimes f_v,w\otimes f_w]] \\
=[v\otimes f_v,f(v,w)(v\otimes f_w+w\otimes f_v)] \\
=f(v,w)[v\otimes f_v,(v+w)\otimes f_{v+w}-v\otimes f_v-w\otimes f_w]\\
=f(v,w)f(v,v+w)(v\otimes f_{v+w}+(v+w)\otimes f_v)\\
-f(v,w)f(v,w)(v\otimes f_w+w\otimes f_v)\\
=f(v,w)^2(v\otimes f_{v+w}+(v+w)\otimes f_v-v\otimes f_w-w\otimes f_v)\\
=2f(v,w)^2\ v\otimes f_v.
\end{array}
$$

Since $\mathfrak{s}_f$ is linearly spanned by its elements of the form $w\otimes f_w$, we find $v\otimes f_v$ to be extremal
and  $g(v\otimes f_v,w\otimes f_w)=f(v,w)^2$.

Clearly $v\otimes f_v$ is pure, if and only if $v\not\in \rad(f)$.
\end{proof}

\begin{lem}\label{forcomp}
Let $v,w,u\in V$. Then $v\otimes f_w+w\otimes f_v\in S_f$ and
$g(v\otimes f_w+w\otimes f_v,u\otimes f_u)=2(f(v,u)\cdot f(w,u))$.
\end{lem}

\begin{proof}
That follows directly from the equality
$$(v+w)\otimes f_{(v+w)}=v\otimes f_v+v\otimes f_w+w\otimes f_v+w\otimes f_w$$
and part (c) of Lemma \ref{rank1extremal}. 
\end{proof}

\begin{lem}\label{fnondegthengnondeg}
The symplectic form $f$ is nondegenerate, if and only if the extremal form $g$ is nondegenerate.
\end{lem}

\begin{proof}
If $g$ is nondegenerate, then by Lemma \ref{rank1extremal}(c) also $f$ is nondegenerate.

Now assume  $f$ to be nondegenerate. 
Let $x\in \mathfrak{s}_f$  be in this radical of $g$.
This element $x$ can be written as a finite sum of scalar multiples of tensors 
$v\otimes f_{v}$, where $v\in V$.
Taking a basis $\{v_1,\dots,v_n\}$ for a nondegenerate subspace of $V$
containing all these vectors $v$, we see that  
 $x$ can be written as 
$$x=\displaystyle\sum_{1\leq i\leq j} \lambda_i v_i\otimes f_{v_i}+\displaystyle\sum_{1\leq i<j\leq n}\lambda_{ij}(v_i\otimes f_{v_j}+v_j\otimes f_{v_i})$$
for some scalars $\lambda_i, \lambda_{ij}\in \mathbb{F}$.

Then, as follows from Lemmas \ref{rank1extremal} and \ref{forcomp} ,
$$
\begin{array}{rl}
g(x,u\otimes f_u)=&
g(\displaystyle\sum_{1\leq i\leq n} \lambda_i v_i\otimes f_{v_i}+\displaystyle\sum_{1\leq i<j\leq n}\lambda_{ij}(v_i\otimes f_{v_j}+v_j\otimes f_{v_i}),u\otimes f_u)\\
=&\displaystyle\sum_{1\leq i\leq n} \lambda_i f(v_i,u)^2+
\displaystyle\sum_{1\leq i<j\leq n}2\lambda_{ij} f(v_i,u)\cdot f(v_j,u)\\
=&0\\
\end{array}$$
for each $u$ in $V$.
So, if we take $u$ to be a vector with $f(v_i,u)= 1$ and $f(v_j,u)=0$ for all $j\neq i$, we find $\lambda _i=0$.
Taking for $u$ a vector with $f(v_i,u)=f(v_j,u)=1$ for some $1\leq i< j\leq n$
and $f(v_k,u)=0$ for all $k\neq i,j$, we find $\lambda_{ij}=0$.
But that implies that $x=0$.
So, the radical of $g$ is trivial and $g$ is nondegenerate.
\end{proof}

\begin{prop}
Suppose $(V,f)$ is a nondegenerate symplectic space over the field $\mathbb{F}$
of characteristic $\neq 2$.
Then the Lie  algebra $\mathfrak{s}_f$ is simple.
\end{prop}

\begin{proof}
Suppose $f$ is nondegenerate.
Let $\mathfrak{i}$ be a nontrivial ideal of $\mathfrak{s}_f$.
Fix a nonzero element $i\in \mathfrak{i}$.
As $g$ is nondegenerate by the above proposition, and
$\mathfrak{s}_f$ is generated by the extremal elements 
$v\otimes f_v$, where $v\in V$,
we find at least one element $x=v\otimes f_v$ with $g(x,i)\neq 0$.
But then $2g(x,i)x=[x,[x,i]]\in \mathfrak{i}$ and hence $x\in \mathfrak{i}$.

Applying the above with $i$ equal to $v\otimes f_v$,
we find that all elements of the form $w\otimes f_w$, where $f(v,w)\neq 0$,
are in $\mathfrak{i}$.
Repeating the argument with these elements as $i$ shows that
all tensors of the form $u\otimes f_u$ are in $\mathfrak{i}$.
But then  $\mathfrak{i}=\mathfrak{s}_f$.

So, we can conclude that $\mathfrak{s}_f$ is simple.
\end{proof}

\begin{prop}\label{puretensors}
Suppose $f$ is nondegenerate.
Then the extremal elements in the Lie algebra $\mathfrak{s}_f$ are
exactly the pure tensors of the form $\lambda v\otimes f_v$, $0\neq v\in V$, $\lambda \in \MF^*$.
\end{prop}

\begin{proof}
Let $x$ be an extremal element.
As $f$ is nondegenerate, so is $g$.

As the pure tensors $v\otimes f_v$ with $v\in V$ linearly span $\mathfrak{s}_f$,
there is an element $y=\lambda v\otimes f_v$ with $g(x,y)=1$.
But then $\mathrm{exp}(x,1)y=y+[x,y]+x=\mathrm{exp}(y,1)x$.
So $x=\mathrm{exp}(y,-1)\mathrm{exp}(x,1) y$.
As by Proposition \ref{auto} we have $\mathrm{exp}(x,-1)\mathrm{exp}(y,1)\in\mathrm{Aut}(\mathfrak{s}_f)$,
we find also $x$ to be a scalar multiple of a pure tensor.
\end{proof}

\begin{exam}\label{exam2dim}
Suppose $(V,f)$ is a nondegenerate symplectic space of dimension at least $2$.
Suppose $W$ is a $2$-dimensional subspace of $V$ spanned by the vectors $v_1,v_2$.
Then the extremal elements  $(\lambda v_1+\mu v_2)\otimes f_{\lambda v_1+\mu v_2}$ all lie  in the $3$-dimensional subspace of $\mathfrak{s}_f$ spanned by the three elements $v_1\otimes f_{v_1}, v_2\otimes f_{v_2}$ and $v_1\otimes f_{v_2}+v_2\otimes f_{v_1}$. The extremal points corresponding to the extremal elements in this three space form a conic, and hence also an oval, of  points $\langle (\alpha,\beta,\gamma)\rangle$ (with respect to the given basis) defined  by the quadratic equation $\alpha\beta=\gamma^2$.

Of course, for hyperbolic $W$ this is in line with Proposition \ref{sl2lines}, but it is also true for singular $W$. 
\end{exam}

In the example below we focus on the case that $W$ is a $3$-dimensional
subspace of $V$ on which the form $f$ does not vanish. The Lie subalgebra of $\mathfrak{s}_f$ generated by the extremal elements $w\otimes f_w$, where $w\in W$,
plays a crucial role in this paper.

\begin{exam}\label{exam4}
Now suppose that $(V,f)$ is a nondegenerate $4$-dimensional symplectic space
with basis $e_1,\dots, e_4$ such that $f(e_1,e_3)=f(e_2,e_4)=1$ and $f(e_i,e_j)=0$ for $\{i,j\}\neq \{1,3\},\{2,4\}$.
Let $W\subset V$ be the $3$-dimensional subspace spanned by
$e_1, e_2$, and  $e_3$, and denote by $\mathfrak{s}$ the subalgebra of $\mathfrak{s}_f$ spanned by the elements $v\otimes f_v$, where $v\in W$.
 Then  $e_2\otimes f_{e_2}$ is nontrivial and in the center of $\mathfrak{s}$. 

Note that in the Lie subalgebra $\mathfrak{s}_0$ of $\mathfrak{s}_{f|_W}(W\otimes W^*)$ generated by the elements $v\otimes f_v$ where $v\in W$, we have $e_2\otimes f_{e_2}=0$, since $f_{e_2}$ is constantly zero on $W$. We have $Z(\mathfrak{s})=\langle e_2\otimes f_{e_2}\rangle$ and $\mathfrak{s}_0\cong \mathfrak{s}/Z(\mathfrak{s})$.

In both $\mathfrak{s}$ and $\mathfrak{s}_0$ we find that the elements
 \begin{align*}
v\otimes f_v-(v+\lambda e_2)\otimes f_{v+\lambda e_2}=-\lambda(e_2\otimes f_v+v\otimes f_{e_2})-\lambda^2 e_2\otimes f_{e_2}
\end{align*}
for $\lambda\in \MF$
generate an  
ideal $\mathfrak{i}$ of $\mathfrak{s}$ and $\mathfrak{i}_0$ in $\mathfrak{s}_0$, respectively, which is, modulo the center $\langle e_2\otimes f_{e_2}\rangle$,  isomorphic to the dual natural $2$-dimensional
module for   $\mathfrak{s}/\mathfrak{i}\simeq \mathfrak{sl}_2(\MF)$
and  $\mathfrak{s}_0/\mathfrak{i}_0\simeq \mathfrak{sl}_2(\MF)$, respectively.
Indeed, $\mathfrak{i}/\langle e_2\otimes f_{e_2}\rangle$ is isomorphic to
$\mathfrak{i}_0=\{e_2\otimes f_w\mid w\in W\}$.

We notice that both $\mathfrak{s}$ and $\mathfrak{s}_0$ can be generated by a {\em symplectic triple}, i.e., a triple of elements $x,y,z$ with $[x,y]\neq 0$ and $[y,z]\neq 0$, $z\not\in \langle x,y\rangle$ and $[x,z]=0$. 
For example, we can choose the extremal elements
\begin{align*}
x:= e_1\otimes f_{e_1},\ y:=e_3\otimes f_{e_3},\ z:=(e_1-e_2)\otimes(f_{e_1}-f_{e_2})
\end{align*}
generating both $\mathfrak{s}$ and $\mathfrak{s}_0$.

The  pure tensors in $\mathfrak{s}$ and $\mathfrak{s}_0$ 
are extremal. They are scalar multiples of elements of the form
\begin{align*}
s:=(\alpha e_1+\beta e_2+\gamma e_3)\otimes (\alpha f_{e_1}+\beta f_{e_2}+\gamma f_{e_3}),
\end{align*}
with $\alpha, \beta,\gamma \in \MF$.  \\
Note that all pure tensors commuting with $x=e_1\otimes f_{e_1}$ are  scalar multiples of elements of the form
\[ (\alpha e_1+\beta e_2)\otimes(\alpha f_{e_1}+\beta f_{e_2})=\alpha^2e_1\otimes f_{e_1} +\beta^2e_2\otimes f_{e_2}+\alpha\beta(e_1\otimes f_{e_2}+e_2\otimes f_{e_1}).\]

In $\mathfrak{s}$ these elements span a  $3$-space $U_x$ with basis  $\{e_1\otimes f_{e_1}, e_2\otimes f_{e_2}, e_1\otimes f_{e_2}+e_2\otimes f_{e_1}\}$, and their coefficients $\alpha,\beta,\gamma\in \MF$  with respect to this basis satisfy the  quadratic equation $\alpha\beta=\gamma^2$. Therefore the corresponding $1$-spaces form a conic, so in particular an oval,  inside the projective plane on $U_x$. The same holds for any other choice of extremal $x$ not in the center $\langle e_2\otimes f_{e_2}\rangle$.

In $\mathfrak{s}_0$ the space spanned by these elements is $2$-dimensional and
they lie in  all but one of the projective points of the corresponding projective line. 

\medskip

We finish this example by showing that  all extremal elements
in  $\mathfrak{s}$ and  $\mathfrak{s}_0$ are   scalar multiples of pure tensors.
(As the form $f$ is degenerate on $W$ we can not use Proposition \ref{puretensors}.)
We start with $\mathfrak{s}_0$.
As every extremal element of $\mathfrak{s}_0$ stays extremal in $\mathfrak{s}_0/\mathfrak{i}_0$, which is isomorphic with $\MSL$,
we see that, up to a scalar, every extremal element $x$ in $\mathfrak{s}_0$ is of the form
$$x=v\otimes f_v+e_2\otimes f_w$$
for some $v\in \langle e_1,e_3\rangle$ and $w\in W$.
Now let $u\in W$, then
$$
\begin{array}{rl}
[x,[x,u\otimes f_u]]=&
[v\otimes f_v+e_2\otimes f_w,[v\otimes f_v+e_2\otimes f_w,u\otimes f_u]]\\
=&
[v\otimes f_v+e_2\otimes f_w,f(v,u)(v\otimes f_u+u\otimes f_v)+\\
&+f(w,u)e_2\otimes f_u]\\
=&
2f(v,u)^2v\otimes f_v+2f(w,u)f(v,u) e_2\otimes f_v+\\
&+f(v,u)f(w,v)e_2\otimes f_u.\\
\end{array}
$$ 
 
As this has to be a scalar multiple of $x$ for each $u\in W$, we find
$v-w$ in the radical of $f$.
But then 
$$\begin{array}{rl}
x=&
v\otimes f_v+e_2\otimes f_w\\
=&v\otimes f_v+e_2\otimes f_v\\
=&v\otimes f_v+e_2\otimes f_v+v\otimes f_{e_2}\\
=&(v+e_2)\otimes f_{v+e_2}.\\
\end{array}$$
So, $x$ is a scalar multiple of a pure tensor.

As extremal elements in $\mathfrak{s}$, that are not central,
map modulo the center $\langle e_2\otimes f_{e_2}\rangle$ to extremal
elements in $\mathfrak{s}_0$, 
such extremal element of  $\mathfrak{s}$ is of the form 
$x+z$, where $x$ is a scalar multiple of a pure tensor  and $z$ a central element.
But as $$[x+z,[x+z,y]]=[x,[x,y]]$$ is a scalar multiple of 
$x$ and of $x+z$ for all $y\in \mathfrak{s}$,
we can conclude that $z=0$. So, also in $\mathfrak{s}$
all extremal elements  are scalar multiples of pure tensors.
\end{exam}

\begin{defi}
By $\mathfrak{sp}_3(\MF)$ we denote the Lie algebra  $\mathfrak{s}$ from the example above.

By $\mathfrak{psp}_3(\MF)$ we denote the Lie algebra  $\mathfrak{s}_0$ from the example above. Notice that $\mathfrak{psp}_3(\MF)$ is isomorphic to
$\mathfrak{sp}_3(\MF)$ modulo its center.
\end{defi}

We will now consider a point-line geometry $\Gamma$, called the {\em $\MSL$-geometry}, 
on the set of  
extremal points of $\mathfrak{s}_f$.
The points of this geometry are the pure extremal points. The lines of this
geometry are the subsets of pure extremal points inside a subalgebra
generated by two noncommuting extremal points.

Clearly, this geometry is isomorphic to $HSp(V,f)$, the {\em geometry of hyperbolic lines} of $(V,f)$, whose point set is the set of $1$-spaces of $V$ outside the radical of $f$, and whose lines are the subsets of points inside a hyperbolic
$2$-space of $V$, i.e., a 2-dimensional subspace on which $f$ induces a nondegenerate form. See \cite{Cuy91}.

The  $\MSL$-geometry of $\mathfrak{sp}_3(\MF)$ (and of $\mathfrak{psp}_3(\MF)$)
is then isomorphic to the geometry $HSp(W,f|_W)$ where $(W,f|_W)$ is a $3$-dimensional
symplectic space with nontrivial form $f|_W$.
This geometry is called a {\em symplectic plane}.
It is isomorphic to a projective plane from which a point and all the lines through that point are removed and hence also called a {\em dual affine plane}.

\section{The $\MSL$-geometry }\label{sec:geom}

In this section we start with the proof of our main result, Theorem \ref{MainThm}.
However we consider a slightly more general situation, which will allow us to use induction at various points in our proofs.

In the following, $\MG$ denotes a 
Lie algebra over the field $\MF$ of characteristic $\neq 2$, generated by its set of pure extremal elements $\widehat E$ and  equipped with the extremal form $g$. Let $\widehat\ME$ be the set of pure extremal points of the Lie algebra $\MG$.

Let $x$ and $y$ be two pure extremal points that generate an $\MSL$.
Then the  {\em $\MSL$-line}  on  $x$ and $y$ is defined to be the set of all extremal points in the subalgebra $\langle x,y\rangle\cong\mathfrak{sl}_2$ (see ~\ref{sl2lines}). Denote by $\mathcal{L}(\MG)$, or $\mathcal{L}$ for short, 
the set of all $\MSL$-lines.

Now consider the point-line space $\Gamma(\MG):=(\widehat{\mathcal{E}},\mathcal{L})$. So in $\Gamma(\MG)$, denoted  by $\Gamma$ if it is clear what Lie algebra we refer to, the  points are the pure 
extremal points and two points $x,y\in \widehat\ME$ are on a line if and only if $x$ and $y$ generate an $\MSL$. We call $\Gamma$ the {\em $\MSL$-geometry} of $\mathfrak{g}$.

\begin{lem}
In the $\MSL$-geometry of $\g$ any two points are on at most one line. 
\end{lem}

\begin{proof}
See Lemma \ref{sl2lines}.
\end{proof}

On $\widehat\ME$ we define the relation $\perp$ by:
 \[\ x\perp y\Leftrightarrow [x,y]=0.\]
If $X\subseteq \widehat\ME$, then $X^\perp$ denotes the set $\{y\in \widehat\E\mid x\perp y \mathrm{\ for\ all\ } x\in X\}$. For $x\in \widehat\E$ we often write $x^\perp$ for $\{x\}^\perp$.

A {\em subspace} $X$ of $\Gamma$ is a subset $X$ of $\widehat\E$ with the property that any line meeting $X$ in at least two points is contained in $X$.

As the intersection of any collection of subspaces is again a subspace, we can define the {\em subspace  generated by} a 
subset $Y$ of $\widehat\E$ to be the intersection
of all subspaces containing $Y$.

We often identify  a subspace $X$ with the point-line space
$\Gamma(X):=(X,\mathcal{L}_X)$ where $\mathcal{L}_X=\{\ell\in \mathcal{L}\mid \ell\subseteq X\})$.

A subspace $X$ of  $\Gamma$ is called {\em nondegenerate} 
if it is connected and for any pair of elements $x,y\in X$ with $x^{\perp}\cap X=y^{\perp}\cap X$, it follows that $x=y$.

From now on we assume that $\mathfrak{g}$
is generated by a subset $\E$ of $\widehat\E$ such that
\begin{enumerate}[(A)]
\item any two elements $x,y\in \E$ commute or generate a subalgebra of $\g$ isomorphic to $\MSL(\mathbb{F})$;
\item for any three elements $x,y,z\in \E$ with $[x,y]\neq 0$, there is an
element $u\in \E$ contained in $\langle x,y\rangle$ commuting with $z$.
\item $\E$ is a connected subspace of $\Gamma$ which linearly spans $\g$. 
\end{enumerate}

The set of extremal elements in $\g$ that span points of $\E$ is denoted by $E$.
As in the previous section, a {\em symplectic triple of extremal elements} is a triple $(x,y,z)$ of elements in $\widehat E$ with
$[x,y]\neq 0 \neq [y,z]$, $z\not\in\langle x,y\rangle$ and $[x,z]=0$.
A {\em symplectic triple of extremal points} is a triple $(x,y,z)$ of extremal points
containing elements $x_1\in x$, $y_1\in y$ and $z_1\in z$ forming a symplectic triple $(x_1,y_1,z_1)$.
  
\begin{prop}\label{symtrip}
A symplectic triple $(x,y,z)$ of pure extremal elements of the Lie algebra $\MG$ generates either a subalgebra isomorphic to $\mathfrak{sp}_3(\mathbb{F})$, in which case it is of dimension $6$, or to $\mathfrak{psp}_3(\mathbb{F})$ of dimension $5$.

Under this isomorphism $x$, $y$ and $z$ are mapped onto scalar multiples of pure tensors
of $\mathfrak{sp}_3(\MF)$ or 
$\mathfrak{psp}_3(\mathbb{F})$, respectively.
\end{prop}

\begin{proof}
Let $(x,y,z)$ be a symplectic triple and $\mathfrak{s}$ the subalgebra generated by $x$, $y$ and $z$.

Consider the six elements $x$, $y$, $z$, $[x,y]$, $[y,z]$ and $[x,[y,z]]$.
Notice that after rescaling we can assume that $g(x,y)=g(z,y)=1$ and $g(x,z)=0$.

We will prove that the subalgebra $\langle x,y,z\rangle$ is linearly spanned by these six elements and that their multiplication table is uniquely determined.

By associativity of $g$, we have
$g(a,[b,c])=0$ for all choices of $a,b,c\in \{x,y,z\}$.
So, $[x,[x,[y,z]]]=2g(x,[y,z])x=0$.  
Moreover, as $[x,z]=0$ the Jacobi identity implies that  $[x,[y,z]+[z,[x,y]]=0$
and we have  
$[z,[x,[y,z]]]=[z,[z,[y,x]]]=2g(z,[y,x])z=0$.
Furthermore, using these identities and the Jacobi identity repeatedly, we find 

$$\begin{array}{rl}
[y,[x,[y,z]]]
=&-[x,[[y,z],y]]-[[y,z],[y,x]]\\
=&[x,2y]+[[z,[y,x]],y]+[[[y,x],y],z]]\\
=&2[x,y]-[y,[z,[y,x]]]+[-2y,z]\\
=&2[x,y]-[y,[x,[y,z]]-2[y,z].\\ 

\end{array}$$
So, $[y,[x,[y,z]]]=[x,y]-[y,z]$.
And,

$$\begin{array}{rl}
[[x,y],[y,z]]
=&- [[y,[y,z],x]-[[[y,z],x],y]\\
=&-2[y,x]+[[[x,y],z],y]\\
=&2[x,y]-[[z,y],[x,y]]-[[y,[x,y]],z]\\
=&2[x,y]-[[x,y],[y,z]]+2[y,z],\\
\end{array}$$
which implies that $[[x,y],[y,z]]=[x,y]+[y,z]$.

But then
$$\begin{array}{rl}
[[x,y],[x,[y,z]]]
=&-[x,[[y,z],[x,y]]]-[[y,z],[[x,y],x]]\\
=&[x,[x,y]+[y,z]]-[[y,z],-2x]]\\
=&2x-[x,[y,z]]\\
\end{array}$$ and
$$\begin{array}{rl}
[[y,z],[x,[y,z]]]
=&-[[z,[x,[y,z]]],y]-[[[x,[y,z]],y],z]\\
=&[[y,[x,[y,z]]],z]\\
=&-[[x,y],z]-[[y,z],z]\\
=&[x,[y,z]]-2z.\\
\end{array}$$

These computations indeed imply that
the subspace of $\MG$ linearly spanned by these six elements is closed under multiplication. Moreover, the multiplication table
of these six  elements  is completely determined.
In particular,  $\mathfrak{s}$ has dimension at most $6$.

The algebra $\mathfrak{sp}_3(\MF)$ is $6$-dimensional and generated by a symplectic triple, see Example \ref{exam4}. But then
the above implies that  $\mathfrak{sp}_3(\MF)$ has a basis satisfying the same 
multiplication table as the above six elements spanning $\mathfrak{s}$.
So, $\mathfrak{s}$ is isomorphic to a quotient of $\mathfrak{sp}_3(\MF)$.
In particular, it is isomorphic to $\mathfrak{sp}_3(\MF)$ or $\mathfrak{psp}_3(\MF)$.
Moreover, this isomorphism
can be chosen to map $x,y$ and $z$ onto pure tensors.
\end{proof}

\begin{prop}\label{symplecticplane}
Two intersecting $\mathfrak{sl}_2$-lines inside the subspace $\E$ of $\Gamma$
generate a subspace isomorphic to a symplectic plane.
\end{prop}

\begin{proof}
Consider  two intersecting $\MSL$-lines $l$ and $m$ inside $\E$. 
There exists an intersection point $z\in\mathcal{E}$ with $z\in l\cap m$ and we can find points $x\in l$ and $y\in m$ such that $(x,z,y)$ is a symplectic triple.

We have seen in ~\ref{symtrip} that a symplectic triple in a Lie algebra $\MG$ generates a $\mathfrak{sp}_3(\MF)$ subalgebra or its central quotient $\mathfrak{psp}_3(\MF)$. Moreover, the extremal elements in $x$, $y$ and $z$ are mapped
to scalar multiples of pure tensors.
As the extremal elements not in the center of $\mathfrak{sp}_3(\MF)$ or
 $\mathfrak{psp}_3(\MF)$ are scalar multiples of pure tensors,
see Example ~\ref{exam4},  we find that $x$, $y$ and $z$ generate a subspace of $\Gamma$ isomorphic to a symplectic plane and contained in $\E$.
\end{proof}

\begin{cor}\label{perpintersectsline}
If $x$ is an extremal point and $\ell$ and $\mathfrak{sl}_2$-line inside $\E$,
then $\ell$ meets $x^\perp$ in a point or is contained in $x^\perp$.  
\end{cor}

\begin{proof}
If $\ell$ is not contained in $x^\perp$, there is an $\mathfrak{sl}_2$-line $m$ on $x$ meeting $\ell$. Inside the symplectic  plane spanned by $m$ and $\ell$ we find a unique point on $\ell$ inside $x^\perp$.
\end{proof}

\begin{lem}\label{nondegen}
If $\mathrm{rad}(g)=\{0\}$, then  $\E$ is nondegenerate.
\end{lem}

\begin{proof}
By assumption $\E$ is connected.
Now consider two extremal points $x,y\in \ME$ with  $x^{\perp}\cap \E=y^{\perp}\cap \E$
and an extremal point $z\in \E$ not in $x^\perp$.
We can find nonzero extremal elements $x_0$, $y_0$ and $z_0$ in $x$, $y$, and $z$ with $g(x_0, z_0)=g( y_0, z_0)$. It follows
\[g(x_0-y_0, z_0)=0.\]

Now suppose $v$ is an arbitrary extremal point in $\E$ and $0\neq v_0\in v$.

If $v\in x^\perp$, then 
\[g(x_0- y_0, v_0)=0.\]

If $v$ is not in $x^\perp$ nor in $z^\perp$, 
then, by Corollary \ref{perpintersectsline}, there is an extremal point $u\subseteq \langle z,v\rangle$ with $u\in x^\perp=y^\perp$.
But then $v_0$ is in the linear span
of $u_0$, $z_0$ and $[u_0,z_0]$, where $0\neq u_0\in u$.
So there are $\alpha,\beta,\gamma\in \mathbb{F}$ with $v_0=\alpha u_0+\beta z_0+\gamma [u_0,z_0]$. Using associativity of $g$ we find 
\[
\begin{array}{rl}
g( x_0- y_0, v_0)
=&g(x_0-y_0,\alpha u_0+\beta z_0+\gamma [u_0,z_0])\\
=&\gamma g([x_0,u_0],z_0)-\gamma g([y_0,u_0],z_0)\\
=&0.
\end{array}\]

It remains to consider the case where $v\in z^\perp$ but not in $x^\perp$.
But then $x,z$ and $v$ are contained in a symplectic plane
all whose points $w$ not in $z^\perp$ satisfy $g(x_0-y_0,w)=0$.
As $v$ is contained in the linear span of these points, we also find
in this case that 
$g(x_0- y_0, v_0)=0$

So, for each extremal element $v_0$ in an extremal point $v$, it follows that $g(x_0 -y_0,v_0)=0$. As, $\g$ is spanned by its extremal elements, we find $x_0- y_0\in \rad(g)=\{0\}$, and we conclude that 
$x_0=y_0$ and $x =y$. 
\end{proof}
 
Geometries in which any two intersecting lines generate a subspace isomorphic to a symplectic plane  have been studied by Cuypers in \cite{Cuy91}.
Using  the main result of \cite{Cuy91},  we obtain the following.

\begin{satz}\label{PhiThm} 
The subspace $\Gamma(\E)$ of $\Gamma$ is isomorphic to the geometry $HSp(V,f)$
of hyperbolic lines of a symplectic space $(V,f)$ over $\MF$. This isomorphism is denoted by
\begin{align*}\label{isophi} \varphi : \Gamma(\E)\overset{\cong}\longrightarrow HSp(V,f).
\end{align*}

The form $f$ is nondegenerate if $\mathrm{rad}(g)=\{0\}$.
\end{satz}

\begin{proof}
If $\Gamma(\E)$ contains a single line, then $\MG$ is isomorphic to
$\mathfrak{sl}_2(\MF)$ and $\Gamma$ is isomorphic to $HSp(2,\MF)$.

So, assume that $\Gamma(\E)$ contains at least two lines.
By assumption, $\Gamma(\E)$ is connected. By \ref{symplecticplane}, we moreover know that in $\Gamma(\E)$ two points are on at most one line, and any pair of intersecting lines is contained in a symplectic plane. Our assumption that $\mathbb{F}$ is not of characteristic $2$ implies that $|\mathbb{F}|\geq 3$ and hence guarantees, by Proposition \ref{sl2lines}, that we have a line with more than $3$ points in $\mathcal{L}_\E$. Then the main result of \cite{Cuy91} can be applied 
and we conclude that $\Gamma(\E)$ is isomorphic to the geometry of hyperbolic lines in a symplectic space. 

As each symplectic plane can be coordinatized by the commutative field $\MF$,
we obtain that $\Gamma(\E)$ is isomorphic to   $HSp(V,f)$ for some symplectic space over $\MF$.
\end{proof}

Concretely, the geometric structure of $HSp(V,f)$  translates to the following:

Let $x,y\in \E$ and  $ \varphi(x)=p, \varphi(y)=q$ be distinct points and $x_1\in x$ and $y_1\in y$ nontrivial extremal elements, then
\begin{align*}
p,q &\text{ are on a hyperbolic line in } (V,f) \Leftrightarrow x \not\perp y \Leftrightarrow g(x_1,y_1)\neq 0.
\end{align*}
So the hyperbolic lines in $HSp(V,f)$ are the lines obtained from the $\mathfrak{sl}_2\text{-lines}$ in $(\ME, \mathcal{L}_\E)$.
The second type of lines in the symplectic space $(V,f)$, the singular lines, correspond to commuting extremal points. Indeed, for two distinct points $ \varphi(x)=p, \varphi(y)=q$ and $x_1\in x$ and $y_1\in y$ nontrivial extremal elements, we have
\begin{align*}
p,q &\text{ are on a singular line in } (V,f) \Leftrightarrow x\perp y\Leftrightarrow g(x_1,y_1)= 0.
\end{align*}

For later use, we need a name  for the equivalent of the singular lines in the symplectic geometry for elements in $\ME$.

\begin{remark} \label{pollines}
Suppose $\E$ is nondegenerate and $x\neq y\in \ME$ with $x\perp y$. 
Then the {\em polar line} through  $x$ and $y$ is the set $(\{x,y\}^{\perp})^{\perp}$. 
The set of all polar lines is denoted by $\mathcal{P}_\E$. By $\mathbb{P}(\ME)$ we denote the point-line space $(\ME, \mathcal{L}_\E \cup \mathcal{P}_\E)$.

With the previous construction and the result of Theorem \ref{PhiThm}, we find,
in case $\E$ is nondegenerate, that the isomorphism $\varphi$ extends uniquely to an isomorphim
\[\hat\varphi: \mathbb{P}(\ME)\overset{\cong}\longrightarrow \MP(V).\]
\end{remark}


\section{The uniqueness of the Lie product}
 \label{sec:unique}

Consider a Lie algebra $\MG$ over the field $\mathbb{F}$ of characteristic not $2$, generated by its set $\widehat\ME$ 
of extremal points,  and with  extremal form $g$. 
 
As in the previous section, assume that $\E$ is a subset of $\widehat\E$
such that
\begin{enumerate}[(A)]
\item any two elements $x,y\in \E$ commute or generate a subalgebra of $\g$ isomorphic to $\MSL(\mathbb{F})$;
\item for any three elements $x,y,z\in \E$ with $[x,y]\neq 0$, there is an
element $u\in \E$ contained in $\langle x,y\rangle$ commuting with $z$.
\item $\E$ is a  connected subspace of $\Gamma$ which linearly spans $\g$.
We additionally assume that $\E$ is nondegenerate.  
\end{enumerate}

The goal of this section is to show that, up to a scalar, the Lie product
$[\cdot,\cdot]$ of $\mathfrak{g}$ is the unique Lie product on the vector space $\mathfrak{g}$, whose $\MSL$-geometry on $\E$ coincides with that of $\mathfrak{g}$.

So assume $[\cdot,\cdot]_1$ is a second Lie product on the vector space 
$\mathfrak{g}$.
Let $\g_1$ be the Lie algebra defined by $[\cdot,\cdot]_1$.

We assume that 
\begin{enumerate}[(a)]
\item the elements of $E$ are extremal with respect to $[\cdot,\cdot]_1$;
\item for $x,y\in E$ we have  $[x,y]=0\Leftrightarrow [x,y]_1=0$;
\item $\E$ is a subspace of $\Gamma(\g_1)$.
\end{enumerate}

Our goal is now to prove that $[\cdot,\cdot]_1$ is a scalar multiple of $[\cdot,\cdot]$.

\begin{lem}
Let $L$ be an $\MSL$-line of $\Gamma$ contained in $\E$, then $L$ is also
an $\MSL$-line in $\Gamma(\g_1)$.
\end{lem}

\begin{proof}
Let $x,y\in\E$ be two points on the line $L$ of $\Gamma$.
Then, restricted to $\E$, we find $L=\{x,y\}^{\perp\perp}$.

But since the perp relation $\perp$ on $\E$ is the same for both
$[\cdot,\cdot]$ and $[\cdot,\cdot]_1$, the $\MSL$-line $M$ on $x$ and $y$
in $\Gamma(\g_1)$ is contained in $L$.  
But then the linear subspaces spanned by  $L$ and $M$ in $\g$ are the same,
implying that $L$ and $M$ are equal. See \ref{sl2lines}.
\end{proof}

\begin{lem}\label{hyppairs}
Let $x,y\in E$ be noncommuting and  
generating a subalgebra $\mathfrak{h}$ of $\MG$. 
Then  there is a $\gamma\in \MF^*$ such that for all
$v,w\in \mathfrak{h}$ we have $[v,w]_1=\gamma[v,w]$.
\end{lem}

\begin{proof}
Without loss of generality we can assume $g(x,y)=1$.
The subalgebra $\mathfrak{h}$ of $\g$ is isomorphic to $\mathfrak{sl}_2$.
Its extremal points are the $1$-spaces spanned by
elements $a x+b y+c [x,y]$ satisfying the equation $ab=c^2$. 
See \ref{sl2lines}.

Also inside $\g_1$ 
the elements $x$ and $y$ generate a
subalgebra $\mathfrak{h}_1$ isomorphic to $\MSL$. Moreover, as follows from the above lemma, the linear subspaces
$\mathfrak{h}$ and $\mathfrak{h}_1$ are equal.
So
\[ [x,y]_1=\alpha x+ \beta y+\gamma [x,y]
\]
for some fixed $\alpha,\beta$ and $\gamma\neq 0$ in $\MF$.
But that implies that the extremal elements \[ y+\lambda [x,y]_1+\lambda^2 g_1(x,y) x=
(\lambda^2g_1(x,y) + \lambda \beta)x+( 1+\lambda\alpha)y+\lambda^2\gamma^2[x,y]
\] 
of $\mathfrak{h}_1$, are also extremal in $\mathfrak{h}$ 
and hence the equation
\[ 
 ( 1+\lambda\alpha)(\lambda^2g_1(x,y) + \lambda \beta) = \lambda^2\gamma^2
\]
is satisfied for each $\lambda\in \mathbb{F}$.

If the field $\mathbb{F}$ contains $3$ elements, then we easily find $\alpha=\beta=0$. If the field contains more than $3$ elements, then the fact 
that the  equation is
a cubic equation in $\lambda$ with more than $3$ solutions implies immediately that $\alpha=\beta=0$.
But that implies that for all $v,w\in \mathfrak{h}$ we have $[v,w]_1=\gamma[v,w]$.
\end{proof}

\begin{lem}\label{symptripproduct}
Let $(x,y,z)$ be a symplectic triple in $\ME$
generating  a subalgebra $\mathfrak{s}$ of $\MG$ isomorphic to
$\mathfrak{(p)sp}_3(\MF)$.

Then there is a scalar $\gamma\in \MF^*$
such that for any two elements $v,w\in \mathfrak{s}$ we have
$[v,w]_1=\gamma [v,w]$.  
\end{lem}

\begin{proof}
Let $\mathcal{S}$ be the set of extremal points in the symplectic plane
generated by $x,y$ and $z$.
For any subset $\mathcal{T}$ of $\mathcal{S}$ we denote by
$E_{\mathcal{T}}$ the set of extremal elements whose span
is in $\mathcal{T}$.

To prove the lemma, it suffices to show that
there exists a $\gamma\in \MF$ such that for all
$v,w\in E_{\mathcal{S}}$ we have $[v,w]_1=\gamma[v,w]$.

Let $L$ be any line of the symplectic plane on $\mathcal{S}$.
Then by Lemma ~\ref{hyppairs} there is an $\gamma_L\in\MF^*$
with $[v,w]=\gamma_L[v,w]_1$ for all $v,w\in E_L$.
Suppose $L, M$ are two lines in the symplectic plane on $\mathcal{S}$.
We will prove that $\gamma_L=\gamma_M$.

Let $p$ be a point on $L$ but not on $M$ and let $q,r,s$ be three distinct points on $M$ collinear with $p$, such that $s\in L$. Denote the line   through $p$ and $q$ by $Q$ and through
$p$ and $r$ by $R$.  By $t$ we denote the unique point on $M$ not collinear to $p$.
Let $p_1,q_1,r_1$ and $s_1$ be extremal elements in $p,q,r$ and $s$, respectively, such that $0\neq q_1+r_1+s_1=t_1\in t$.
Then 

$$
\begin{array}{ll}
0&=[p_1,t_1]\\
 &=[p_1,q_1+r_1+s_1]\\
 &=[p_1,q_1]+[p_1,r_1]+[p_1,s_1]
\end{array}
$$
and, moreover
$$
\begin{array}{ll}
0&=[p_1,t_1]_1\\
 &=[p_1,q_1+r_1+s_1]_1\\
 &=[p_1,q_1]_1+[p_1,r_1]_1+[p_1,s_1]_1\\
 &=\gamma_Q[p_1,q_1]+\gamma_R[p_1,r_1]+\gamma_L[p_1,s_1]
\end{array}
$$

This implies
that
$$(\gamma_L-\gamma_Q)[p_1,q_1]+(\gamma_L-\gamma_R)[p_1,r_1]=0.$$

If $[p_1,q_1]$ and $[p_1,r_1]$ are linearly independent,
we find $\gamma_L=\gamma_Q=\gamma_R$.
If $[p_1,q_1]$ and $[p_1,r_1]$ are linearly dependent,
then $\gamma_Q=\gamma_R$, as then $[p_1,[p_1,q_1]]_1=\gamma_Q[p_1,[p_1,q_1]]$
but also $[p_1,[p_1,q_1]]_1=\gamma_R[p_1,[p_1,q_1]]$.
With a similar argument, but permuted $L,Q$ and $R$, we find
$\gamma_L=\gamma_Q=\gamma_R$.

This shows that for all lines $L'$ on $p$ inside $\mathcal{S}$ we have 
$\gamma_{L'}=\gamma_L$.
But by connectedness of the symplectic plane, we find this to be true for any line $L'$ in $\mathcal{S}$, which clearly proves the statement of the lemma.
\end{proof}

\begin{prop}\label{uniquelieproduct}
There is a scalar $\gamma\in \MF^*$
such that for any two elements $v,w\in \mathfrak{g}$ we have
$[v,w]_1=\gamma [v,w]$.  
\end{prop}

\begin{proof}
Let $L$ be a line of $\mathcal{L}_\E$.
By Lemma ~\ref{hyppairs}
there is a $\gamma\in \MF^*$ with $[\cdot,\cdot]_1=\gamma[\cdot,\cdot]$
restricted to the subalgebra generated by $L$. 

The above Lemma \ref{symptripproduct} implies that for any
line $M$ of $\mathcal{L}_\E$ intersecting $L$ we have that $[\cdot,\cdot]_1=\gamma[\cdot,\cdot]$
restricted to the subalgebra generated by $M$.
But then connectedness of $\E$ implies that 
$[\cdot,\cdot]_1=\gamma[\cdot,\cdot]$
restricted to the subalgebra generated by any line $N$ of $\mathcal{L}_\E$,
which, as the points of $\mathcal{E}$ linearly span $\g$, 
clearly implies the proposition.
\end{proof}

We can apply the above in the following situation.

\begin{prop}\label{Iso}
Let $(V,f)$ be a nondegenerate symplectic space.
Suppose there exists a semi-linear isomorphism  between the vector spaces
$\MG$ and  $S_f$, inducing an isomorphism
$\Gamma(\E)\cong \Gamma(\mathfrak{fsp}(V,f))(\cong HSp(V,f))$, then
$\MG$ is isomorphic to $\mathfrak{fsp}(V,f)$. \end{prop}

\begin{proof}
After identifying the underlying vector spaces  of 
$\mathfrak{g}$ and $S_f$ of $\mathfrak{fsp}(V,f)$, 
we find $\Gamma(\E)$ to be equal to $\Gamma(\mathfrak{fsp}(V,f))$.
Thus, we can apply Proposition ~\ref{uniquelieproduct} and find that,
up to a scalar multiple, the two Lie products of 
$\mathfrak{g}$ and $\mathfrak{fsp}(V,f)$ are the same and hence these
Lie algebras are isomorphic.
Indeed, if $[\cdot,\cdot]_1=\gamma [\cdot,\cdot]$, for some $\gamma\in \mathbb{F}^*$,
then scalar multiplication with $\gamma^{-1}$ provides the isomorphism:

$$[\gamma^{-1}v,\gamma^{-1}w]_1=\gamma^{-2}[v,w]_1=\gamma^{-2}\gamma[v,w]=\gamma^{-1}[v,w]$$
for all $v,w\in  \g$.
\end{proof}

\section{The embedding}
\label{sec:embedding}

Before we begin with the last steps of the identification of a Lie algebra $\MG$ satisfying the conditions of Theorem \ref{MainThm}, 
we review the previous results.

Let $\MG$ be a Lie algebra  over the field $\MF$ with  $\chara \MF\neq 2$, generated by its set of  extremal elements $E$. Assume that the radical 
of the extremal form $g$ is trivial, and,  moreover, that  
conditions (a) and (b) of Theorem \ref{MainThm} hold.
By Theorem \ref{PhiThm} each of the connected components of the point-line geometry 
$\Gamma(\mathfrak{g})=(\E,\mathcal{L})$ is then isomorphic  to the geometry of hyperbolic lines  $ HSp(V,f)$ for some nondegenerate symplectic space $(V,f)$. 

If $\Gamma(\mathfrak{g})$ consists of more than one connected component, then 
$\mathfrak{g}$ is a direct sum of the subalgebras generated by the extremal points in the various components. So, we can consider the case where 
$\Gamma(\mathfrak{g})$ consists of a single connected component
isomorphic  to   $ HSp(V,f)$ for some nondegenerate symplectic space $(V,f)$. 

The previous section, in particular Proposition \ref{Iso}, implies that
if the vector space $\g$ is isomorphic to the vector space
underlying $\mathfrak{fsp}(V,f)$, and this isomorphism induces an isomorphism between the two $\MSL$-geometries,  
then the Lie algebras $\g$ and  $\mathfrak{fsp}(V,f)$ are isomorphic.

In this section it is our goal to establish the existence of  such an isomorphism of vector spaces in the finite dimensional case.

Again we consider a slightly more general situation.
From now on we assume that $\mathfrak{g}$
is generated by its set of extremal elements $\widehat E$ and that
$E$ is a subset of $\widehat E$ such that
\begin{enumerate}[(A)]
\item any two elements $x,y\in E$ commute or generate a subalgebra of $\g$ isomorphic to $\MSL(\mathbb{F})$;
\item for any three elements $x,y,z\in E$ with $[x,y]\neq 0$, there is an
element $u\in E$ contained in $\langle x,y\rangle$ commuting with $z$.
\item $\E$ is a nondegenerate connected subspace of $\Gamma$ which linearly spans $\g$. 
\end{enumerate}

Here $\E$ denotes the set of extremal points spanned by elements of $E$.

Now, by Theorem \ref{PhiThm} the subspace $(\E,\mathcal{L}_\E)$ is isomorphic to
$HSp(V,f)$ for some nondegenerate symplectic space $(V,f)$.
Our main tool for establishing a linear bijection between the vector spaces $\g$ and $S_f$ is the following result due to Schillewaert and Van Maldeghem \cite{SvM13}:

\begin{satz}[\cite{SvM13}]\label{SvM}
Let $U$ be a vector space of dimension $n\geq 3$ 
over a field $\mathbb{F}$ of order at least $3$ equipped with a nondegenerate symmetric bilinear form $b$, and $W$ a vector space of dimension $d$ over a field $\mathbb{K}$.
Let $\cV$ be an injective map from the point set $P$ of $\mathbb{P}(U)$ into 
the point set of $\mathbb{P}(W)$ such that:

\begin{enumerate}[$(a)$]
\item $\cV(P)$ spans $\mathbb{P}(W)$.
\item for any line $\ell$ of $\mathbb{P}(U)$, 
the image $\cV(\ell)=\{\cV(x)\mid x\in \ell\}$
spans a plane of $\mathbb{P}(W)$ 
intersecting $\cV(P)$ in precisely  $\cV(\ell)$;
the points of $\cV(\ell)$ form an oval in the plane.
\item $d\geq \frac{1}{2}n(n+1)$.
\end{enumerate}

Then $d=\frac{1}{2}n(n+1)$ and $\mathbb{F}\cong \mathbb{K}$,
and there exists a semi-linear bijection $\phi: S_b(U\otimes U^*) \rightarrow W$
such that
$\cV(\langle u\rangle)=\langle \phi(u\otimes b_u)\rangle$ for all $u\in U$.
\end{satz}

We notice that the above result is a reformulated (and somewhat weaker) version
of Theorem 2.3 of \cite{SvM13}.
Schillewaert and Van Maldeghem phrase their result in terms of 
quadric Veroneseans and symmetric matrices.

We want to apply the above result to prove that there is a semi-linear bijection between the vector space of $S_f$ and $\g$.
We first will apply it to the case that $U=V$ and $W=S_f$.
Indeed, for finite dimensional $V$ we find that the map $\cV:\mathbb{P}(V)\rightarrow\mathbb{P}( S_f)$ defined by
$$\cV(\langle v\rangle)=\langle v\otimes f_v\rangle$$
for all $0\neq v\in V$, does satisfy the conditions of Theorem \ref{SvM},
see Example \ref{exam2dim}. 
So, Theorem \ref{SvM} implies that there is a semi-linear bijection between $S_f$ and $S_b(V\otimes V^*)$ for  any symmetric nondegenerate bilinear form $b$ on $V$, mapping pure tensors to pure tensors.
This can be made concrete in the following way.

Suppose $(V,f)$ is a nondegenerate symplectic space of finite dimension.
Fix a basis $B$ of $V$ and consider the element $s=\Sigma_{u\in B}u\otimes f_u$. Then the action of $s$ on $V$ is invertible.
Now consider the bilinear form $b:V\times V\rightarrow \mathbb{F}$
given by $$b(v,w)=f(s(v),w)=\Sigma_{u\in B}f(u,v)f(u,w)$$
for all $v,w\in V$.
Clearly, $b$ is nondegenerate and symmetric.
So,  the linear map $\psi:S_f\rightarrow S_b$ defined by
$$\psi(v\otimes f_v)=v\otimes b_v$$
for all $v\in V$, 
is an isomorphism mapping the pure tensors of $S_f$ to the pure tensors of $S_b$.

This implies the following.

\begin{cor}\label{corSvM}
Let $(V,f)$ be a nondegenerate symplectic space of finite dimension $n\geq 4$ over a field $\mathbb{F}$. Let
$\cV$ be a map of the point set $P$ of $\mathbb{P}(V)$ into the point set of 
$\mathbb{P}(W)$, where $W$ is a vector space of dimension $d$ over  a field $\mathbb{K}$,
such that the conditions $(a)$, $(b)$ and $(c)$ of Theorem \ref{SvM} hold.

Then $d=\frac{1}{2}n(n+1)$ and $\mathbb{F}\cong \mathbb{K}$,
and there exists a semi-linear bijection $\phi: S_f \rightarrow W$
such that
$\cV(\langle v\rangle)=\langle \phi(v\otimes f_v)\rangle$ for all $v\in V$.

\end{cor}

In the following two lemmas  we show that the embedding 
of $\mathbb{P}(\ME)\simeq \mathbb{P}(V)$ into $\mathbb{P}(\g)$ satisfies the conditions of the corollary, provided  that $V$ is finite dimensional. 

\begin{lem}\label{V1*}
Let $\ell$ be a line of $\mathbb{P}(\E)$. Then the points of  $\ell$ span a  plane $\pi$ of $\MP(\MG)$, with $\mathcal{E}(\MG)\cap \pi=\ell$ forming an oval. 
\end{lem}

\begin{proof}
We have to distinguish two cases for $\ell$; it is an $\mathfrak{sl}_2$-line or a polar line.
Let us first consider the case that $\ell$ is an $\mathfrak{sl}_2$-line. 
Then $\langle \ell \rangle$  is indeed a $3$-dimensional subalgebra isomorphic to
$\mathfrak{sl}_2(\mathbb{F})$ and the extremal points in $\langle x ,y\rangle$ form a conic, and hence an oval, as we have seen in \ref{sl2lines}.

Now assume that $\ell$ is a polar line. 
We have to prove that the linear span of $\ell$ is a $3$-dimensional subspace of $\mathfrak{g}$ meeting $\ME$ just in $\ell$.
And, moreover, that the points on $l$ form a conic and hence an oval in this $3$-space.

Fix two points $x$ and $y$ on $\ell$.
Let $z$ be a point in $\ME$ such that $(x,z,y)$ is a symplectic triple. 
The subalgebra $\langle x,y,z\rangle$ is isomorphic to $\mathfrak{sp}_3(\mathbb{F})$ or its central quotient $\mathfrak{psp}_3(\mathbb{F})$, see Proposition \ref{symtrip}.

The point $z$ is collinear with all but  one point, say $a$, on the polar line $\ell$. 
Clearly  $\ell\setminus\{a\}\subseteq \langle x,y,z\rangle$.
As we have seen in Example \ref{exam4}, the points of $\ell\setminus\{a\}$ are all contained in a subspace of 
$\mathfrak{g}$ of dimension $3$ if $\langle x,y,z\rangle\simeq \mathfrak{sp}_3(\MF)$ and of dimension $2$ if $\langle x,y,z\rangle\simeq \mathfrak{psp}_3(\MF)$.
In the first case they are all but one of the points of a conic (the missing point being the center of  $\langle x,y,z\rangle$)
and in the second case all but one of the points of the $2$-space.

As we can assume that $\dim(V)\geq 4$, we can  
consider a second point $z'$ with  $(x,z',y)$ being another symplectic triple
but this time with $z'$ collinear with $a$, but not with some $a'\neq a$ in $\ell$.
As above, we find that all points of $\ell\setminus \{a'\}$
are  contained in a subspace of 
$\mathfrak{g}$ of dimension $3$  or $2$.
Moreover, in the first case they are all but one of the points of a conic (the missing point being the center of  $\langle x,y,z'\rangle$)
and in the second case all but one of the points of the $2$-space.

If $\ell\setminus \{a\}$  spans a $2$-space, then 
this $2$-space is contained in $\langle x,y,z'\rangle$, 
and  we find at least three extremal points in it that are not commuting with $z'$.
But this implies that also  $ \ell\setminus \{a'\}$  spans a $2$-space
and $ \ell\setminus \{a\}$ and $ \ell\setminus \{a'\}$ span the same $2$-space.
In particular, $a$ is contained in this $2$-space. But since $a\in \langle x,y,z\rangle\cong\mathfrak{psp}_3(\MF)$ and the center of $\mathfrak{psp}_3(\MF)$ is trivial, it follows $[a,z]\neq 0$,
contradicting that $a$ is not collinear to $z$.

Hence $ \ell\setminus \{a\}$  and $\ell\setminus \{a'\}$ both span a $3$-dimensional subspace.

Let $c$ be  the center of $\langle x,y,z\rangle$.
Then, as we can see in $\Gamma(\g)\cong HSp(V,f)$, 
every element $u\in \ME$ that commutes with $a$ also commutes with all points of a polar line
on $a$ that meets  $\langle x,y,z\rangle$  in at least three points. 
As $c$ is in the span of these points, we find $[u,c]=0$.

Let $c_1$ be a nonzero element of $c$ and $a_1$ be a nonzero element
of $a$ and fix $\lambda,\mu\in \MF$, not both $0$, such that $g(z',\lambda c_1+\mu a_1)=0$.
As also $g(u, \lambda c_1+\mu a_1)=0$ for all $u\in \ME$ with $a\perp u$,
we find that  $g(v, \lambda c_1+\mu a_1)=0$ for all $v\in\langle u^\perp,z'\rangle=\MG$.
This implies that $\lambda c_1+\mu a_1$ is in the radical of $g$ and hence $0$.
But then $a=c$, so $\ell$ is a conic (and hence an oval) in $\langle \ell\rangle$.
Finally we notice that the
extremal points in $\langle \ell\rangle$ are in $\ell$.
See Example \ref{exam4}.
This proves the lemma.
\end{proof}

\begin{lem}\label{V3*}
Suppose $(\ME,\mathcal{L}_\E)\cong HSp(V,f)$ for some nondegenerate symplectic space $(V,f)$ of dimension $2m$.
Then  $\dim \MG\geq m(2m+1)$. 
\end{lem}

\begin{proof}
We prove this by induction on $m$.

If $m=1$, then $(\E,\mathcal{L}_\E)$ is a single $\MSL$-line and $\dim \MG=\dim \MSL=3=1\cdot(2\cdot 1+1)$.

Now suppose the statement of the proposition is true for some $m\in \MN$. 
Then consider the case where $(\E,\mathcal{L}_\E)\cong HSp(V,f)$, 
for some  nondegenerate symplectic space $(V,f)$ of dimension $2(m+1)$. 
We fix two noncommuting extremal points $x,y$ in $\E$  and consider $\MG_0=\langle z\in \ME|[z,x]=[z,y]=0\rangle$. This Lie algebra is generated by $\E_0=\{z\in \E\mid z\perp x,y\}$, a subspace of  $(\E,\mathcal{L}_\E)$  isomorphic to $HSp(V',f')$, where $(V',f')$ is a nondegenerate symplectic space of dimension $2m$. By induction $\MG_0$ has dimension $\geq m(2m+1)$. 
So, we can apply Corollary \ref{corSvM} and find that there is an isomorphism
between the vector space $\g_0$ and $S_{f'}$ mapping extremal elements to pure tensors. But then Proposition \ref{Iso} implies that $\g_0$ is isomorphic to
 $\mathfrak{sp}(V',f')$.

We provide a lower bound for the dimension of $\g$ by computing lower bounds
for various composition factors of the action of $\g_0$ on $\g$.

Consider $\MG_x/(\langle x\rangle +\MG_0)$, where $\MG_x:=\langle z\in \ME |[z,x]=0\rangle $. Each polar line $s$ on $x$ spans a $3$-space in $\MG_x$ which meets $\MG_0$ in at most one point, so $s$ maps to a space of dimension at most $1$ in $\MG_x/(\langle x\rangle +\MG_0)$. Now assume that $s\subseteq \langle \MG_0+x\rangle$. Then the set $C_s(y)$ of elements in $s$ commuting with $y$ contains $\MG_0\cap s$, which is at least $2$-dimensional. But inside the $\mathfrak{sp}_3(\MF)$-subalgebra spanned by $x,y$ and $s$ we see that $C_s(y)$ is a $1$-space, a contradiction. So it follows that indeed $s$ is mapped to a $1$-dimensional subspace in $\MG_x/(\langle x\rangle +\MG_0)$.

Let $\mathfrak{s}:=\mathfrak{sp}_3(\MF)$ as in Example \ref{exam4} be a Lie algebra generated by a symplectic triple, such that $x$ spans the center of $\mathfrak{s}$. Then the intersection of the geometries of  $\mathfrak{s}$ and $\MG_0$ is a hyperbolic line $l$ (so a $3$-space in $\MG_0$) and $\mathfrak{s}$ is mapped to a subspace of dimension at most $6-(3+1)=2$ in $\MG_x/(\langle x\rangle +\MG_0)$. We prove that this subspace is indeed of dimension $2$. 
We use that $\mathfrak{s}\cong N: \MSL$, where $N\cong \MF^{1+2}$ is an ideal isomorphic to a  non-split extension of the natural module for $\MSL$ by a $1$-dimensional center. Note that the elements of $\mathfrak{s}$ that are in $\MSL$ commute with $y$, as stated before. So assume there is an $n\in N$ that commutes with $y$. Clearly $n$ is not in the center of  $N$. But the action of $\MSL$ on $N/\langle x\rangle$ is the action on the natural module, so the images of $n$ under this action will generate the full ideal $N$ and commute with $y$. This implies $[x,y]=0$, a contradiction. So $\mathfrak{s}$ maps to a $2$-dimensional subspace in $\MG_x/(\langle x\rangle +\MG_0)$.

Note that the  geometry of the space spanned by polar lines $l$ on $x$ together with all possible subspaces $\mathfrak{s}$ as above on $x$ is isomorphic to $HSp(V',f')$. 

As follows from the above, this space is naturally embedded into $\MG_x/(\langle x\rangle +\MG_0)$, which therefore has dimension $2m$ (by \cite{Cuy91} and \cite{Lef81}) and is isomorphic to the natural module for $\MG_0$. 

A similar construction of the spaces $\MG_y$ and $\MG_y/(\langle y\rangle+\MG_0)$ leads to similar conclusions.

So  $\MG_x/(\langle x\rangle +\MG_0)$ and $\MG_y/(\langle y\rangle +\MG_0)$ are both $2m$-dimensional and by the above construction natural modules for $\MG_0$. 
 These natural modules are irreducible, and we deduce
\begin{align*}\dim \MG_{2(m+1)}\geq&\dim \langle x,y\rangle +\dim \MG_x/(\langle x\rangle +\MG_0)+ \dim \MG_y/(\langle y\rangle +\MG_0) + \dim \MG_0\\
\geq & 3+2m+2m+m(2m+1)\\
=&2m^2+5m+3\\
=&(m+1)(2(m+1)+1).
\end{align*} 

So, the proposition is proved by induction.
\end{proof}

We are now able to identify finite dimensional $\g$:

\begin{satz}\label{ThmVeronesian}
If $\g$ is finite dimensional, then $\MG\cong \mathfrak{sp}(V,f)$.
\end{satz}
\begin{proof} 
Suppose $\g$ is finite dimensional.
The $\MSL$-geometry $(\E,\mathcal{L}_\E)$ is isomorphic to $HSp(V,f)$. 
As we saw in  Remark \ref{pollines}, this isomorphism, denoted by $\varphi$, extends uniquely to an isomorphism $\hat\varphi$ between $\mathbb{P}(\ME)$ and  $\mathbb{P}(V)$.

The space  $V$ is also finite dimensional. Indeed, as $\E$ spans $\g$, we can find a finite set $\mathcal{B}\subset \E$ spanning $\g$. But then $\varphi(\mathcal{B})^\perp$ is empty. As $(V,f)$ is nondegenerate, this is only possible if $V$ is finite dimensional.  

By \ref{V1*} and \ref{V3*} we can conclude that $\hat\varphi^{-1}$ maps
points from $\mathbb{P}(V)$ to points of $\mathbb{P}(\E)$ and  satisfies 
the conditions (a), (b) and (c) of Corollary \ref{corSvM}.
In particular, there exists a semi-linear invertible  $\phi$ from $S_f$ 
to $\g$ with for all nonzero $v\in V$:
$$\langle\phi(v\otimes f_v)\rangle= \hat\varphi^{-1}(\langle v\rangle).$$
But that means that we can not only identify the vector spaces $S_f$ and $\g$, but, simultaneously, also the $\MSL$-geometries on the extremal points of $\mathfrak{s}_f$ and on $\E$. So, we can  apply Proposition \ref{Iso} and conclude that $\mathfrak{s}_f$ is isomorphic to $\g$.
\end{proof}

\section{Proof of Theorem \ref{MainThm}}

\label{sec:final}

In this final section we will finish the proof of Theorem \ref{MainThm}.

So, suppose that $\mathfrak{g}$ is a simple Lie algebra over
the field $\mathbb{F}$ of characteristic not $2$.
Moreover, assume that $\mathfrak{g}$ is generated by its set of extremal elements $E$ satisfying the conditions (a) and (b) of Theorem \ref{MainThm}.

Let $g$ be the extremal form of $\mathfrak{g}$.
As the radical of $g$ is an ideal in $\mathfrak{g}$,
we find this radical to be trivial.
In particular, all elements in $E$ are pure, see \ref{extremalform}

That $\mathfrak{g}$ is simple, also implies that the 
geometry $\Gamma(\mathfrak{g})$ is nondegenerate.
Hence, $\Gamma(\mathfrak{g})$ is isomorphic to $HSp(V,f)$ for some nondegenerate
symplectic space $(V,f)$ over $\mathbb{F}$. See Theorem \ref{PhiThm}.
Denote this isomorphism by $\varphi$.

If $V$ is finite dimensional, then we can apply the results of the previous section, in particular, Theorem \ref{ThmVeronesian} can be applied, and we find 
$\mathfrak{g}$ to be isomorphic to 
$\mathfrak{sp}(V,f)$. 

So, assume that $V$ is not of finite dimension.
Then let $\mathcal{S}$ be the set of all nondegenerate, finite dimensional 
subspaces of $V$, partially ordered by inclusion. Clearly, this is a directed set.  Then for each $V_0\in\mathcal{S}$ denote by 
$\mathfrak{g}_{V_0}$ the subalgebra of $\mathfrak{g}$ generated by $\E_{V_0}$, the set of those extremal points $x$ for which $\varphi(x)\subseteq V_0$. 

The set $\E_{V_0}$ is closed under $\MSL$-lines and hence a  subspace of 
$\Gamma(\g)$ and of $\Gamma(\g_{V_0})$ isomorphic to $HSp(V_0,f)$.

\begin{lem}\label{dimensions}
Suppose $V_0\in\mathcal{S}$.
Then the Lie algebra $\g_{V_0}$ is finite dimensional and linearly spanned by
the extremal elements in $\E_{V_0}$.
\end{lem}

\begin{proof}
Suppose that $V_0$ is in $\mathcal{S}$. 
Then  $HSp(V_0,f)$ can be generated by a finite set $X$ of points.
Indeed, fix a  vector $0\neq v\in V_0$. 
Then the geometry with as points the  hyperbolic lines on $\langle v\rangle$ and as lines the symplectic planes on $\langle v\rangle$ is an affine space of finite dimension. 
So, there exists a finite set of hyperbolic lines on  $\langle v\rangle$
generating this affine space.
Now consider the set consisting of  the point $\langle v\rangle$
and a single point different from $\langle v\rangle$ on each of these 
generating hyperbolic lines. 
This finite set of points in $HSp(V_0,f)$ generates a subspace 
containing all lines and hence all symplectic planes on $\langle v\rangle$.
But, as each point of $HSp(V_0,f)$ is inside some symplectic plane with $\langle v\rangle$, the whole space $HSp(V_0,f)$ is generated by this finite set of points.

The Lie algebra $\g_{V_0}$ is then 
generated by the corresponding finite set of extremal points
in $\varphi^{-1}(X)$. But then $\g_{V_0}$ is finite dimensional, as follows from \cite[Theorem 4.1]{CSUW01}.

Since $\g_{V_0}$ is generated by the elements in $\E_{V_0}$
and for any two extremal elements $x,y$ inside points of $\E_{V_0}$, the product
$[x,y]$ is contained in the linear span of $\E_{V_0}$,
we find $\g_{V_0}$ to be linearly spanned by the elements in $\E_{V_0}$. 
\end{proof}

The above lemma implies that we can apply
Theorem \ref{ThmVeronesian} to find  $\mathfrak{g}_{V_0}$ to be isomorphic to  $\mathfrak{sp}(V_0,f)$. 

The subalgebras  $\mathfrak{g}_{V_0}$ and $\mathfrak{sp}(V_0,f)$, with
$V_0\in \mathcal{S}$, form isomorphic  local systems for $\mathfrak{g}$ and $\mathfrak{fsp}(V,f)$, respectively. As both $\mathfrak{g}$ and $\mathfrak{fsp}(V,f)$ are equal to the corresponding direct limits, we find them to be isomorphic, which finishes the proof of Theorem \ref{MainThm}.